\documentclass{amsart}
\usepackage{amsmath}
\usepackage{amsfonts}
\usepackage{amssymb}
\usepackage{mathtools} 
\usepackage{xcolor}
\usepackage{graphicx}
\usepackage{tikz-cd}
\newtheorem{Theorem}{Theorem}[section]
\newtheorem{Cor}[Theorem]{Corollary}
\newtheorem{Lemma}[Theorem]{Lemma}
\newtheorem{Proposition}[Theorem]{Proposition}

\newtheorem{Definition}[Theorem]{Definition}
\newtheorem{rem}[Theorem]{Remark}
\newtheorem{example}[Theorem]{Example}

\newcommand{\R}{\mathbb R}
\newcommand{\K}{\mathbb{K}}

\newcommand{\N}{\mathbb N}

\newcommand{\C}{\mathcal{C}}

\newcommand{\F}{\mathcal{F}}
\newcommand{\p}{\mathcal{P}}

\newcommand{\rel}{\mathcal{R}}

\newcommand{\mcc}{\mathcal{C}}

\title[Profinite diffeology]{On the geometry of profinite diffeological spaces}
\author{Anahita Eslami Rad${}^\times$}
\author{Jean-Pierre Magnot${}^{*}$}

\author{Enrique G. Reyes${}^\dagger$}

\address{$^\times$ Department of Mathematics, FaMAF, Universidad Nacional de Cordoba, Medina Allende, Ciudad 
Universitaria , 5000-Cordoba, Argentina}
\address{$^*$  Univ. Angers, CNRS, LAREMA, SFR MATHSTIC, F-49000 Angers, France ;\\ 
Lepage Research Institute, 17 novembra 1, 081 16 Presov, Slovakia;
	\\ and \\  Lyc\'ee Jeanne d'Arc, \\ Avenue de Grande Bretagne, \\ 63000 Clermont-Ferrand, France}
\address{$^\dagger$ Departmento de Matem\'{a}tica y Ciencia de la Computaci\'{o}n, \\
	Universidad de Santiago de Chile \\ Casilla 307 Correo 2,
	Santiago, Chile.}

\email{${}^\times$ anahita.eslami@unc.edu.ar}
\email{$^*$magnot@math.cnrs.fr}
\email{$^*$jean-pîerr.magnot@ac-clermont.fr}
\email{$\dagger$enrique.reyes@usach.cl}
\email{$\dagger$e\_g\_reyes@yahoo.ca}
\begin{document}

	\begin{abstract}
		We consider the class of profinite diffeological spaces, that is, diffeological spaces which diffeologies are deduced by pull-back of diffeologies on finite-dimensional manifolds through a system of projection mappings. This class includes inductive limits of finite-dimensional manifolds, as well as solution spaces of differential relations and spaces that agree with cylindrical approximations. We analyze tangent and cotangent spaces, differential forms, Riemannian metrics, connections, differential operators, Laplacians, symplectic forms, momentum maps and the relation between de Rham and singular cohomology on them, unifying constructions partially developed in various contexts. 
	\end{abstract}
	
	\maketitle
	\vskip 12pt
	\textit{Keywords:} diffeological spaces, profinite geometry.
	
	\textit{MSC (2020): 58A05, 53C15 } 
	\section*{Introduction}

Projective systems are a classical and powerful tool for constructing and analyzing infinite-dimensional spaces by assembling finite-dimensional approximations. Such systems naturally arise in algebraic geometry, {}{infinite dimensional geometry}, topology, and analysis—where examples include inverse systems of manifolds, jet spaces, and spaces of paths or loops equipped with finite evaluations. However, the geometric structure of their projective limits often escapes the framework of smooth manifolds or even {}{Fr\'echet} manifolds.

To address this difficulty, the theory of \emph{diffeological spaces}, originally introduced by J.-M. Souriau \cite{Sou} and later developed extensively by Iglesias-Zemmour \cite{Igdiff}, provides a natural language. Diffeologies allow one to define smooth maps, differential forms, and tangent objects on arbitrary sets equipped with a family of plots—without requiring charts, local trivializations, or topological regularity. This flexibility makes them particularly well-suited to study spaces that arise as limits of smooth structures.

In this paper, we propose a general geometric framework for analyzing projective systems through the lens of diffeology. {}{A similar work has been initiated in \cite{AM99}, in a period when the vocabulary and the technical results on diffeologies and diffeological-like settings were not as developed as now. This delay in time explains the vocabulary that superficailly differs between the two works, and we will comment on the differences and the deep similarities of some of our results, that were already obtained in this paper. Indeed, the actual technical capacities on diffeologie enable one to go much further in the results than whan was possible to do more 25 years ago in the setting of Fr\"olicher spaces \cite{KM}. Since \cite{Ma2006-3}, the two frameworks have been compared, and Fr\"olicher spaces appear as a restricted category of diffeological spaces, see e.g. \cite{BKW2025,GMW2023,Ma2013,Wa}.}
     
     Given {}{an index set $I$, ordered but not totally ordered in general, and} a projective system \( \{X_i, \pi_{ij}\}_{i \leq j \in I} \) of diffeological spaces with smooth bonding maps, we endow the projective limit, {}{in the sense of \cite{AM2016},}   with the \emph{projective diffeology}—the coarsest diffeology making all canonical projections smooth. This construction is fully compatible with the structure of differential forms and tangent {{} bundles,} and allows us to transport geometric data from the finite-dimensional levels to the limit space.

We investigate in particular:
\begin{itemize}
  \item the behavior of differential forms under pullback along the projections,
  \item the identification of tangent and cotangent modules of the limit in terms of the system,
  \item and the conditions under which structures such as vector bundles or differential complexes descend to the limit.
\end{itemize}

This perspective enables us to treat {{} a large class of} infinite-dimensional spaces {{}(at least those underlying technical points in finite dimensional analysis or geometry)} in a controlled and geometrically meaningful way. As a key example, we consider the class of spaces equipped with \emph{cylindrical diffeologies}, i.e., diffeologies generated by finite-dimensional evaluation maps. These arise naturally in analysis on function spaces {{} over finite dimensional Lindel\"off topological manifolds}, where the evaluation at finite sets of parameters provides a canonical system of projections.

Beyond the standard examples of ``towers'' of inclusions indexed by the totally ordered index set $\N,$ we reach examples where the index set is only partially ordered.
A particularly important case is the classical Wiener space \( W = C_0([0,1], \mathbb{R}^n) \), which can be seen as the projective limit of finite-dimensional distributions of Brownian motion. Its cylindrical diffeology encodes the smooth structure relevant to Malliavin calculus and stochastic differential geometry \cite{ELA1999,Mal1997,Wa1984}, and allows for a differential-geometric treatment of forms and cotangent modules. 

The paper is organized as follows: after an introcution to the necesary notions in diffeologies in Section \ref{s:diffeology}, Section~\ref{s:profinite} introduces the projective diffeology. Then classical differential geometric constructions ar discussed: smooth mappings, bundles, group actions, cotangent and tangent spaces, differential forms, (pseudo-)Riemannian and (pseudo-)Hermitian metrics, symplectic forms and non-degeneracy. We finish with {{}illustrating} examples, including algebras {{}of} matrices, jet spaces and the Wiener space. We prove that {{} general} profinite diffeological spaces do not suffer all the possible pathologies of more genral diffeological spaces. For example, the different definitions of the tangent space of a diffeological space turn out to be equivalent for profinite diffeological spaces, polynomials of cylindrical functions separate the points of the space, and {{}standard constructions in symplectic geometry, such as the introduction of momentum maps, can be performed in a classical way.} (compare with the general constructions of \cite{Igdiff}).

	\section{Diffeological preliminaries} \label{s:diffeology}
	This section provides the necessary background on diffeology and vector pseudo-bundles, mostly following 
	\cite{MR2019,Ma2020-3}, but keeping \cite{Igdiff} as our main reference for diffeologies. A 
	complementary non-exhaustive bibliography on these subjects is 
	\cite{BKW2025,BN2005,BT2014,CN,CSW2014,CW2014,Don,DN2007-1,DN2007-2,IgPhD,Leandre2002,pervova,Wa}.
	
	\subsection{Basics of Diffeology}\label{ss:diffeology}
	
	In this subsection we review the basics of the theory of diffeological spaces; in particular, their 
	definition, categorical properties, as well as their induced topology. {{} The main idea of diffeologies (and 
	also of Fr\"olicher spaces, which were defined shortly after, see {\em e.g.} \cite{MR2016} and references
	therein) is to replace the atlas of a classical manifold by other intrinsic objects that enable to define 
	smoothness of mappings in a safe way, considering manifolds as a restricted class of examples. Many such 
	settings have been developed independently. We choose the two settings just mentioned because they possess 
	nice properties such as cartesian closedness, thereby carrying the necessary fundamental properties for e.g. 
	calculus of variations, and also because they are very easy to use in a differential geometric way of 
	thinking. The starting point of diffeological and Fr\"olicher spaces consists in defining families of smooth 
	maps via mild axiomatic conditions that ensure that they have technical features of interest. }
	
	\begin{Definition}[Diffeology] \label{d:diffeology}
		Let $X$ be a set.  A \textbf{parametrisation} of $X$ is a
		map 
		$p \colon U \to X$, in which $U$ is an open subset of Euclidean space (no fixed dimension).  A 
		\textbf{diffeology} $\p$ on $X$ is a set of parametrisations satisfying the following three conditions:
		\begin{enumerate}
			\item (Covering) For every $x\in X$ and every non-negative integer
			$n$, the constant function $p\colon \R^n\to\{x\}\subseteq X$ is in
			$\p$.
			\item (Locality) Let $p\colon U\to X$ be a parametrisation such that for
			every $u\in U$ there exists an open neighbourhood $V\subseteq U$ of $u$
			satisfying $p|_V\in\p$. Then $p\in\p$.
			\item (Smooth Compatibility) Let $p\colon U\to X$ be a parametrisation in $\p$.
		Then for every $n >0$, every open subset $V\subseteq\R^n$, and every
			smooth map $F\colon V\to U$, we have $p\circ F\in\p$.
		\end{enumerate}
		
		A \textbf{diffeological space} $(X,\p)$ is a set $X$ equipped with a diffeology $\p$.
		When the diffeology is understood, we drop the symbol $\p$.
		The parametrisations $p\in\p$ are called \textbf{plots}.
	\end{Definition}
	
	{{}
		\noindent\textbf{Notation.} We recall that $\N^* = \{n \in \N \, | \, n \neq 0\}$ and that 
		$\forall m \in \N^*, \N_m = \{1,...,m\} \subset \N.$}
	
	\begin{Definition}[Diffeologically Smooth Map]\label{d:diffeolmap}
		Let $(X,\p_X)$ and $(Y,\p_Y)$ be two diffeological
		spaces, and let $F \colon X \to Y$ be a map.  We say that $F$ is
		\textbf{diffeologically smooth} if for any plot $p \in \p_X$,
		$$F \circ p \in \p_Y.$$
	\end{Definition}

	Diffeological spaces with diffeologically smooth maps form a category which is complete, co-complete, 
	and in fact, a quasi-topos (see \cite{BH}).  In particular, in this category pull-backs, push-forwards, 
	(infinite) products, and (infinite) coproducts exist. 
	
	We define pull-backs and push-forwards by means of the
	following propositions, after \cite{Sou,Igdiff}.
	
	\begin{Proposition}
		Assume that $(X',\p)$ is a diffeological space.
		{{}Let} $X$ be a set and $f:X\rightarrow X'$ a map.
		The \textbf{pull-back diffeology}  $f^*(\p)$ on $X$ is 
		{{} $$f^*( \p)= \left\{ p: O_p \rightarrow X \, |   f \circ p \in \p \right\}. $$ }
	\end{Proposition}
	
	\begin{Proposition} Assume that $(X,\p)$ is a diffeological space,
		and let $X'$ be a set and $f:X\rightarrow X'$ a map.
		The \textbf{push-forward diffeology} $f_*(\p)$ is the coarsest (i.e. the smallest for inclusion) among the 
		diffologies on $X'$, which contain $f \circ \p.$ 
	\end{Proposition}  
	
	\begin{Definition}
		Let $(X,\p)$ and $(X',\p')$
		be two diffeological spaces.  A map $f : X \rightarrow X'$ is  called  a subduction
		if  $\p' = f_*(\p).$ 
	\end{Definition}
	
	In particular, we have the following constructions.
	
	\begin{Definition}[Product Diffeology]\label{d:diffeol product}
		Let $\{(X_i,\p_i)\}_{i\in I}$ be a family of diffeological spaces.  Then the \textbf{product diffeology} 
		$\p$ on $X=\prod_{i\in I}X_i$ contains a parametrisation $p\colon U\to X$ as a plot if and only if for 
		every $i\in I$, the map $\pi_i\circ p$ is in $\p_i$.  Here, $\pi_i$ is the canonical projection map 
		$X\to X_i$. 
	\end{Definition}
\noindent 	In other words, $\p = \cap_{i \in I} \pi_i^*(\p_i)$ and each $\pi_i$ is a 
	subduction. 
	
	\begin{Definition}[Subset Diffeology]\label{d:diffeol subset}
		Let $(X,\p)$ be a diffeological space, and let $Y\subseteq X$.  Then $Y$ comes equipped with the 
\textbf{subset diffeology}, which is the set of all plots in $\p$ with image in $Y$.
	\end{Definition}

\smallskip

	If $X$ is a finite or infinite-dimensional smooth manifold modelled on a complete locally convex topological 
	vector space, we define its \textbf{nebulae diffeology} as
\begin{equation} \label{aste}
\p_\infty(X) = \left\{ p \in C^\infty(O,X) \hbox{ (in the usual sense) }| O \hbox{ is open in } \R^d, d \in \N^* 
	\right\}.
\end{equation}

	\subsection{Fr\"olicher spaces}
	
	\begin{Definition} ~ 
	
	$\bullet$ A \textbf{Fr\"olicher} space is a triple
		$(X,\F,\mcc)$ such that
		
		- $\mcc$ is a set of paths $\R\rightarrow X$ and $\F$ is a set of functions from $X$ to $\R$;
		
		- A function $f:X\rightarrow\R$ is in $\F$ if and only if for any
		$c\in\mcc$, $f\circ c\in C^{\infty}(\R,\R)$;
		
		- A path $c:\R\rightarrow X$ is in $\mcc$ (i.e. is a \textbf{contour})
		if and only if for any $f\in\F$, $f\circ c\in C^{\infty}(\R,\R)$.
		
		 $\bullet$ Let $(X,\F,\mcc)$ et $(X',\F',\mcc')$ be two
		Fr\"olicher spaces, a map $f:X\rightarrow X'$ is \textbf{differentiable}
		(=smooth) if and only if one of the following equivalent conditions is fulfilled:
		\begin{itemize}
			\item $\F'\circ f\circ\mcc\subset C^{\infty}(\R,\R)$
			\item $f \circ \C \subset \C'$
			\item $\F'\circ f \subset  \F$ 
		\end{itemize}
	\end{Definition}
	
	Any family of maps $\F_{g}$ from $X$ to $\R$ generate a Fr\"olicher
	structure $(X,\F,\mcc)$ by setting, after \cite{KM}:
	\vskip 5pt
	\begin{flushleft}
	
	First step: $\mcc=\{c:\R\rightarrow X \hbox{ such that } \F_{g}\circ c\subset C^{\infty}(\R,\R)\}$
	
	Second step: $\F=\{f:X\rightarrow\R\hbox{ such that }f\circ\mcc\subset C^{\infty}(\R,\R)\}.$
	
\end{flushleft}	
	\vskip 5pt
\noindent We easily see that $\F_{g}\subset\F$. This construction will be useful
	in the sequel to describe in a simple way a Fr\"olicher structure.
	A Fr\"olicher space carries a natural topology, the pull-back topology of $\R$ via $\F$. { In the 
	case of a finite-dimensional differentiable manifold $M$, we can equip $M$ with the Fr\"olicher structure
	determined by the sets
	$$ \C = \{ c : \R\rightarrow M \hbox{ such that } c \in C^{\infty}(\R,M) \}$$ and 
	$$ \F = \{ f : M \rightarrow\R \hbox{ such that } f \in C^{\infty}(M,\R)\}.$$
	The underlying topology
	of this Fr\"olicher structure is the same as the manifold topology, \cite{KM}.} 
	
	Let us now compare Fr\"olicher spaces with diffeological spaces. We equip a Fr\"olicher space $(X,\F,\mcc)$
	with the following diffeology {{}$\p_\infty(\F)$}:
	
	\noindent Let $O \subset \R^p${ be an open subset of a Euclidean space; we set} 
		$$\p_\infty(\F)_O=
		\coprod_{p\in\N}\{\, f : O \rightarrow X; \, \F \circ f \subset C^\infty(O,\R) \quad \hbox{(in
			the usual sense)}\}$$
		and 
		$$ \p_\infty(\F) = \bigcup_O \p_\infty(\F)_O\, ,$$
		where the latter union is extended over all open sets $O \subset \R^n$ for $n \in \N^*.$ 
	 {In analogy with (\ref{aste}), we call $\p_\infty(\F)$ the ``nebulae diffeology", along the lines of 
	 \cite{Ma2013}. }
	With this construction, we get a natural diffeology when
	$X$ is a Fr\"olicher space and we can easily show the following:
	
	\begin{Proposition} \label{Frodiff} \cite{Ma2006-3} 
		Let$(X,\F,\mcc)$
		and $(X',\F',\mcc')$ be two Fr\"olicher spaces. A map $f:X\rightarrow X'$
		is smooth in the sense of Fr\"olicher if and only if it is smooth for
		the underlying nebulae diffeologies. 
	\end{Proposition}
	
	{
	We finish our comparison of the notions 
	of diffeological and Fr\"olicher space following mostly \cite{Ma2006-3,Wa}:}
	
	\begin{Theorem} \label{compl-fro}
		Let $(X,\p)$ be a diffeological space. There exists a unique Fr\"olicher structure
		$(X, \F_\p, \mcc_\p)$ on $X$ such that for any Fr\"olicher structure $(X,\F,\mcc)$ on $X,$ these two 
		equivalent conditions are fulfilled:
		
		(i)  the canonical inclusion is smooth in the sense of Fr\"olicher 
		$(X, \F_\p, \mcc_\p) \rightarrow (X, \F, \mcc)$
		
		(ii) the canonical inclusion is smooth in the sense of diffeologies 
		$(X,\p) \rightarrow (X, \p_\infty(\F)).$ 
		
		\noindent Moreover, $\F_\p$ is generated by the family 
		$$\F_0=\lbrace f : X \rightarrow \R \hbox{ smooth for the 
			usual diffeology of } \R \rbrace.$$
		{{} We say that the {{}Fr\"olicher} structure $(X, \F_\p, \mcc_\p)$ is the \textbf{Fr\"olicher completion} of 
		$\p$ }
	\end{Theorem}
	
	Summarizing we can state, intuitively, that the following implications hold and can be understood 
	alternatively with the vocabulary of subcategories {and functors}:
	\vskip 12pt
	\begin{tabular}{ccccc}
		Smooth manifold  & $\Rightarrow$  & Fr\"olicher space  &{ ${\Rightarrow}$   
		} & Diffeological space 
		\tabularnewline
	\end{tabular}
	\vskip 12pt

	The next remark is based on \cite[p.26, Boman's theorem]{KM}.
	\begin{rem}
		The set of contours $\C$ of the Fr\"olicher space
		$(X,\F,\C)$ \textbf{does not} give us a diffeology, because a diffelogy
		needs to be stable under restriction of domains. In the case of paths in
		$\C$ the domain is always $\R.$ However, $\C$ defines a ``minimal diffeology''
		$\p_1(\F)$ whose plots are smooth parameterizations which are locally of the
		type $c \circ g,$ where $g \in \p_\infty(\R)$  and $c \in \C.$ Within this setting,
		a  map $f : (X,\F,\C) \rightarrow (X',\F',\C')$ is smooth if and only if it is smooth  
		$(X,\p_\infty(\F)) \rightarrow (X',\p_\infty(\F')) $ or equivalently smooth 
		$(X,\p_1(\F)) \rightarrow (X',\p_1(\F')) $ 
		Moreover, any diffeology can be completed into a reflexive one, but not all the diffeologies are 
		reflexive. Indeed, in \cite{Ma2013}, it is shown that, given a {{}Fr\"olicher} space $(X,\F,\mcc),$ that 
		differ from $\p_\infty(\F),$ which is generated by the set of contours $\mcc$. This diffeology is {{}not} reflexive except in particular cases, e.g. when $X = \R$ or $S^1$ (1-dimensional manifolds). A 
		deeper analysis of these implications is given in \cite{Wa}, see e.g. \cite{BKW2025} for a non-exhaustive 
		presentation on this subject. 
	\end{rem}
	
	We also remark that given an algebraic structure, we can define a
	corresponding compatible diffeological structure. For example, a
	$\R-$vector space equipped with a diffeology is called a
	diffeological vector space if addition and scalar multiplication
	are smooth (with respect to the canonical diffeology on $\R$), see \cite{Igdiff,pervova}. An
	analogous definition holds for Fr\"olicher vector spaces. Other
	examples will arise in the rest of the text. {As an application of our foregoing remarks, we note 
	that any complete locally convex topological vector space is a diffeological vector space}. 
	
	\begin{rem} \label{comp}
		Fr\"olicher, $c^\infty$ and {{}G\^ateaux} smoothness are the same notion
		if we restrict to a Fr\'echet context, see \cite[Theorem 4.11]{KM}.
		Indeed, for a smooth map $f : (F, \p_1(F)) \rightarrow \R$ defined
		on a Fr\'echet space with its 1-dimensional diffeology, we have
		that $\forall (x,h) \in F^2,$ the map $t \mapsto f(x + th)$ is
		smooth as a classical map in $\C^\infty(\R,\R)$ and hence, it is
		Gateaux smooth. The converse is obvious.
	\end{rem}

	\begin{Proposition} \label{quotient} 
	Let $(X,\p)$ {{} be} a diffeological	space and $\rel$ an equivalence relation on $X$. Then, there is
		a natural diffeology on $X/\rel$ defined as the push-forward diffeology by the quotient projection
		$X\rightarrow X/\rel$. It is denoted by $\p/\rel$.
	\end{Proposition}

	\begin{example}
		Let $X$ be a diffeological space. Let us {{} denote} by 
		$$
		X^\infty = 
		\left\{ (x_n)_{n \in \N} \in X^\N \, | \, \{n \, | x_n \neq 0 \} \hbox{ is a finite set}\right\} \, .
		$$ 
		This is a diffeological space, as a subset of $X^\N.$ The quotient space $X^\N / X^\infty$ is a 
		diffeological space related to the settings of non standard analysis, see e.g. 
	 \cite{AHKFL1986}.  
	\end{example}
	
	\subsection{Vector pseudo-bundles}\label{section:vpb}
	
	{{} We begin with a precise look at the notion of a fiber bundle in classical (finite-dimensional) 
	differential geometry. Fiber bundles are defined by 
	\begin{itemize}
			\item a smooth manifold  $E$ called total space
			\item a smooth manifold  $X$ called base space
			\item a smooth submersion $\pi: E \rightarrow X$ called fiber bundle projection
			\item a smooth manifold $F$, called typical fiber, such that $\forall x \in X,  \pi^{-1}(x)$ 
			is a smooth submanifold of $E$ diffeomorphic to $F.$
			\item a smooth atlas on $X,$ with domains $U \subset X$ such that $\pi^{-1}(U)$ is an open submanifold 
			of $E$ diffeomorphic to $U \times F.$ We call this atlas a system of local trivializations of the fiber 
			bundle.
		\end{itemize}
It is enough to define a smooth fiber bundle as the quadruple data $(E,X,F,\pi)$, because the definitions of $\pi$ 
and $X$ enable us to find systems of local trivializations. This quadruple setting is often denoted simply 
by the projection map $\pi: E\rightarrow X.$
		
There exist diffeological spaces which carry no atlas, so that the condition of having a system of smooth 
trivializations in a generalization of the notion of fiber bundles is not a priori necessary, even if this 
condition is technically interesting, see \cite[pages 194-195]{MW2017}. 
%
In the diffeology context, the notion of quantum structure was introduced in \cite{Sou} as a generalization of 
principal bundles, and the notion of vector pseudo-bundle was advanced in \cite{pervova} as a generalization of 
vector bundles, {{} see e.g. \cite{Ma2025-1}.} The common idea of these works is the description of fibred objects made of a total 
(diffeological) space $E,$ a diffeological space $X$, and a canonical {{} (diffeologically)} smooth bundle projection 
$\pi: E \rightarrow X$ such that, $\forall x \in X,$ the set $\pi^{-1}(x)$ is endowed with a ({{}diffeologically} smooth) algebraic 
structure. The existence of local trivializations is not assumed. 
	
	\begin{enumerate}
	\item For a diffeological vector pseudo-bundle, the fibers $\pi^{-1}(x)$ are diffeological vector spaces, i.e. 
	vector spaces where addition and multiplication over a diffeological field of scalars (e.g. $\R$ or 
	$\mathbb{C}$) is smooth. We notice that \cite{pervova} only deals with finite-dimensional vector spaces.
	\item For a ``structure quantique'' (i.e. ``quantum structure'') following the terminology of \cite{Sou}, we
	assume that a diffeological group $G$ is acting from the right, smoothly and freely, on a diffeological space 
	$E$. The space of orbits $X=E/G$ equipped with the quotient diffeological structure, defines the 
	base of the quantum structure $\pi: E \rightarrow X$.
	In this picture, each fiber $\pi^{-1}(x)$ is isomorphic to $G.$
	\end{enumerate}
	
	From these two examples, we generalize this picture as follows. 
	
	\begin{Definition}\label{pseu-fib}
		Let $E$ and $X$ be two diffeological spaces and let $\pi:E\rightarrow X$ be a smooth surjective map. Then 
		$(E,\pi,X)$ is a \textbf{diffeological fiber pseudo-bundle} if and only if  $\pi$ is a subduction. 
	\end{Definition}
	
	We do not assume that there exists a typical fiber, in coherence with Pervova's diffeological vector 
	pseudo-bundles. We specialize this definition as follows, along the lines of \cite{Ma2020-3,Ma2025-1}:
	
	\begin{Definition}
		Let $\pi:E\rightarrow X$ be a diffeological fiber pseudo-bundle. Then:
		 $\pi:E\rightarrow X$ is a \textbf{diffeological $\mathbb{K}-$vector pseudo-bundle}, in which
			$\mathbb{K}$ is a diffeological field, if there exists: 
			\begin{itemize}
				\item a smooth fiberwise map $\,\cdot \, :\mathbb{K} \times E \rightarrow E,$
				\item a smooth fiberwise map $\, + :E^{(2)} \rightarrow E$ where 
				$$E^{(2)} = \coprod_{x \in X} \{(u,v) \in E^2\, | \, (u,v)\in \pi^{-1}(x)\}$$ 
				equipped by the pull-back diffeology of the canonical map $E^{(2)} \rightarrow E \times E$, 
			such that $\forall x \in X, $ $(\pi^{-1}(x),+,\cdot)$ is a diffeological $\mathbb{K}-$vector 
			{ space}. 
			\end{itemize}
			We say that $E$ is a \textbf{ diffeological vector bundle} is $E$ is a diffeological vector 
			pseudo-bundle whose fibers are all isomorphic as diffeological vector spaces. 
			
				
				
		\end{Definition}
		
		\subsection{Tangent and cotangent spaces}
		The review \cite{GMW2023} identifies five main definitions of tangent space that generalize the classical 
		tangent space of a finite-dimensional manifolds, starting from one of the following two viewpoints:
		\begin{itemize}
			\item The tangent space of a finite-dimensional manifold is defined by 1-jets of paths.
			\item  The tangent space of a finite-dimensional manifold is defined by pointwise derivations of 
			$\R-$valued maps. 
		\end{itemize}
		When replacing ``finite-dimensional manifold'' by ``diffeological space'', these two notions do not 
		coincide. Even more, unlike what happens in the $c^\infty-$setting \cite{KM}, the natural mapping 
		$$`` \hbox{1-jets} \quad \longrightarrow \quad \hbox{pointwise derivations}''$$ 
		may not be injective, and the space of 1-jets at one point may be only a cone, and not a vector space. 
		For this reason, the choice of one of the five generalizations is made taking into account the  
		application at hand.
		
{The \textbf{internal tangent cone} extends straightforwardly the definition of the tangent space of a 
		smooth 	manifold in terms of germs of paths (compare with the kinematic tangent space in \cite{KM}), and 
		it coincides with the internal tangent cone first described for Fr\"olicher spaces in \cite{DN2007-1}. See 
		the extensive description given in \cite{GMW2023} }
		
	There are other tangent spaces in the category of diffeological spaces: the diff-tangent space which is 
	defined in \cite{Ma2020-3} and that will be described in next parts of the exposition,
		the \textbf{internal tangent space} is defined in \cite{He1995,CW2014}, based on germs of paths. {{} This 
		second definition is necessary, and the internal tangent space differ from the internal tangent cone. 
		Indeed, spaces of germs do not carry intrinsically a structure of abelian group. This remark was first 
		formulated in the context of Fr\"olicher spaces, see \cite{DN2007-1}, and see \cite{CW2014} for the 
		generalization to diffeologies. For this reason, one can complete the tangent cone into a vector space, 
		called internal tangent space. This was performed in \cite{CW2014} via mild considerations on colimits in 
		categories. From another viewpoint, the \textbf{external tangent space} is spanned by derivations, and one 
		can define the cone of derivations which are defined by germs of paths \cite{Ma2013} as well as its vector 
		space completion in the space of derivations following \cite{GW2020}.  
			For finite-dimensional manifolds these tangent spaces coincide. For an extended presentation, we refer 
			to the review \cite{GMW2023}.
			
	The definition of the cotangent space $T^*X,$ as well as the definition of the algebra of differential forms 
	{{}$\Omega^*(X,\K)$,} does not carry so many ambiguities: it is defined through pull-back on each domain of plots, 
	see \cite{Igdiff} and e.g. \cite{Ma2013} for a comprehensive exposition. 
			
	\section{Profinite diffeological spaces and its limits} \label{s:profinite}
	
	In this section we follow \cite{AM2016} for the presentation of projective systems, but we complete it {{}with the help of} diffeological notions {{} while \cite{AM99} considers convenient calculus uh Fr\"olicher spaces}.
	  
	\begin{Definition} 
	A set $A$ equipped with a partial ordering $\leq$ is called a	{\bf directed set} if 
	$$ \forall (J,K) \in A^2, \exists R \in A (J \leq R \, \wedge \, K \leq R ) .$$
	\end{Definition}
	
	\begin{Definition} 
	Let us consider a family $${{}E=}\left\{E_J\right\}_{J \in A}$$ of (non-empty) sets
	labelled by the elements of a non-empty directed set A, called index
	set. We call such a family a \textbf{projective} or \textbf{inverse family} $(E,\pi)$ if 
	there exists a family of surjective maps
	$$\pi=\left\{{{}\pi_J^K:E_K \rightarrow E_J} \, | \, J \leq K\right\}$$ 
	satisfying the {\bf consistency property}
	$$ \forall (J,K,L) \in A^3,\, J \leq K \leq L \Rightarrow {{}\pi_J^L =  \pi_J^K \circ \pi_K^L}\; ,$$ 
	and such that 
	$$\forall K \in A,\, \pi_K^K  = Id_{E_K}.$$
	\end{Definition}
	
Projective (or, inverse) families can be considered in the category of sets or topological spaces, as well as in 
the category of diffeological spaces. They can be also defined in the category of smooth manifolds, 
but in this case the definition below, for instance, would not always make sense without the extended category of 
diffeologies. This is what we describe after.

	\begin{Definition}
	Let $(E,\pi)$ be a projective family. Let 
$$E^A = \left\{(x_J)_{j \in A} \in \prod_{J \in A} E_J \, | \, \forall J \leq K, \, x_J = \pi_J^K(x_K)\right\}.$$
	We call $E^A$ the \textbf{product limit set} of $(E,\pi).$ 
\end{Definition}

\begin{Definition}
	Let $(E,\pi)$ be a projective family.
	\begin{itemize}
		\item (see e.g. \cite{AM2016}) Let us assume that $\forall J \in A,$ $E_J$ is a topological space. Then 
		$(E,\pi)$ is a topological projective family if the family of maps $\pi$ is a family of continuous maps. 
	\item Let us assume that $\forall J \in A,$ $E_J$ is a diffeological space. Then $(E,\pi)$ is a diffeological
	 projective family if the family of maps $\pi$ is a family of smooth maps. 
	\end{itemize}
	\end{Definition}

\begin{example}
	Let $I$ be a non-empty set and let $\left\{E_i\right\}_{i \in I}$ ba a family of non-empty sets (``spaces'') 
	labelled by the elements of a non-empty set $I.$ Let $A = \mathfrak{P}_F (I)$ to be the set of finite subsets 
	of $I.$ {We define the partial ordering $\leq$ on $A$ as the inclusion,} and for $J \in A,$ we set 
	$E_J = \prod_{i \in J} E_i.$  {{} We also set 
		\begin{itemize}
			\item $E = \{E_J\}_{J \in A}$
		\item if $J \subset K,$ $\pi_J^K$ is the canonical projection from $E_K$ to $E_J.$
		\end{itemize}  
	
	If the sets $E_i$ are equipped with a topology (resp. diffeology), }
	then $(E,\pi)$ is a topological (resp. diffeological) projective family. 
\end{example}

	\begin{Definition}
		A \textbf{profinite} family is a projective family $(E,\pi)$ such that 
		\begin{enumerate}
			\item $\forall J \in A,$ $E_J$ is a finite-dimensional manifold.
			\item $\forall (J,K) \in A^2,$ if $J \leq K,$ there exists an injective map 
			      $i_K^J : E_J \rightarrow E_K$ such that $$\pi_J^K \circ i_K^J = Id_{E_J}.$$ 
            \item {$\forall (I,J,K) \in A^3 $ such that $I \leq J \leq K,$ 
                  $i^{J}_K \circ i^I_J = i^I_K.$}
		\end{enumerate} 
	We note by $(E,\pi,i)$ such a profinite family. A profinite family is called topological (resp. diffeological) 
	if 
	\begin{itemize}
		\item $(E,\pi)$ is a topological (resp. diffeological) projective family;
		\item the family of maps $i=\{ i_K^J:E_J \rightarrow E_K | J \leq K \}$ is a family of continuous 
		(resp. smooth) maps.
	\end{itemize}
	\end{Definition}
	
	{Profinite families admit not only product limits, but also inductive limits. We define the
	latter below, after some preliminaries.}
	
\begin{rem} 
	Let $(E,\pi, i)$ be a profinite family. We define a relation  $\sim$ on the coproduct $\coprod_{I \in A} E_I$ 
as follows: $x_J \in E_J \sim x_K \in E_K$ if and only if $\exists L \leq K, J$ and $\exists x_L \in E_L$ such
that
$$
 i^L_K(x_L)= x_K \mbox{ and } i^L_J(x_L)= x_J\, .
 $$  

This relation is obviously reflexive and symmetric. 
	We define a relation of equivalence on the directed set $A,$ that we note $\mathcal{R},$ by the following 
	property (with obvious notations): 
	
	$x_J \mathcal{R} x_L$ if there exists a finite sequence of indexes $(I_1,...,I_n)$ and a finite sequence  
	$(x_1,...x_n)$ such that 
	\begin{itemize}
		\item $I_1 = J,$ $I_n = K$
		\item $x_1 = x_J,$ $x_n = x_K$
		\item $\forall i \in \N_{n-1},$ $x_i \sim x_{i+1}.$ 
		\end{itemize}
	
\end{rem}

\begin{Definition} \label{sec}
	{{} Le $A$ be a directed set.} $\mathcal{S}  { \in}  \mathfrak{P}(A)-\{\emptyset\}$ is a \textbf{section} of $A$ if
	\begin{itemize}
		\item[(a)] $\forall (I,J)\in \mathcal{S}^2, I \neq J \Leftrightarrow (I \nleq J \wedge J \nleq I)$ 
		{
		(i.e., $S$ is an anti-chain for the partially ordered set $(A , \leq)$).}
		\item[(b)] $\forall I \in A, \exists S \in \mathcal{S}, (I \leq S \vee S \leq I).$
	\end{itemize}
We note by $Sec(A)$ the set of sections of $A$ and by {{}$Sec_F(A)$} the set of finite sections of $A$ 
(which may be void).
\end{Definition}.

\begin{Lemma}
$A$ is totally ordered if and only if $Sec(A)=A,$ that is, any section is of cardinality 1, and any element of $A$ 
defines a section.
\end{Lemma}
\begin{proof} {The left-to-right implication is clear. Conversely, let us assume that all
elements of $A$ determine a section and that the cardinality of every section is 1. This means that if 
$P,Q \in A$, we have two possibilities for the set $S_0 = \{ P,Q\}$: either (a) of Definition \ref{sec}
fails, or (b) of Definition \ref{sec} fails. If (a) fails, we obtain that $P$ and $Q$ are comparable.  
Assume that (b) fails. Then, there exists $I_0 \in A$ such that for all $S \in S_0$, $I_0$ and $S$
are incomparable. But this means that $\{I_0\}$ is not a section, a contradiction. We conclude that $A$ is
totally ordered, as claimed.}
\end{proof}

\begin{Definition}
	Let $\mathcal{S} \in Sec(A).$ Let $E_\mathcal{S}$ be the subset of $E^A$ defined by the following property: 
	
	$$ (x_I)_{I \in A} \in E_\mathcal{S} \Leftrightarrow \forall I \in A, \exists S \in \mathcal{S}, \left\{
	\begin{array}{lcl}x_I =  i^S_I(x_S) & \hbox{ if } & S \leq I \\
		x_I = \pi^S_I(x_S) & \hbox{ if } & I \leq S \end{array} \right.$$ 
	and we define the \textbf{ inductive limit set} $E_A$ as $$E_A = \bigcup_{\mathcal{S} \in Sec(A) } E_S.$$
\end{Definition}

\begin{example}
	Let $A = \N.$ and let $\leq$ be the usual total order of $\N.$ 
 Let $(E,\pi)$ be a projective family such that $\forall (j,k)\in \N^2, j \leq k \Rightarrow E_j \subset E_k$ and 
 $\pi_j^k$ is a projection from $E_k$ to $E_j.$ Then the {inductive} limit set of the projective 
 family $(E,\pi)$ is $$\cup_{n \in \N} E_n.$$ 
	This is the case when $E_n = \R_n[X],$ and where, for $n \in \N^*,$ the projection $\pi^n_{n-1}$ reads as
	$$\pi^{n}_{n-1}\left(\sum_{k=0}^n a_k X^k\right)= \sum_{k=0}^{n-1} a_k X^k.$$
	In that case, $E_A = \R[X]$ and $E^A = \R[[X]].$
\end{example}

\section{On the filter defined by sections and profinite distances}


We assume here that 
$$\forall J \in A, E_J \hbox{ is equipped with a distance } d_J$$
such that $$\forall (J,K) \in A^2, J \leq K \Rightarrow d_J = d_K \circ \left( i_K^J \times i_K^J \right).$$   

\begin{Definition}
	On $E_A$ and $E^A$ we define
	\begin{itemize}
		\item $d^\infty_E = \sup_{J \in A} \frac{d_J}{1+d_J}$
		\item if $\mu$ is a finite measure on $A,$ {{} we define the measurable map} $J \in A \mapsto 
		\frac{d_J(x,y)}{1+d_J(x,y)}$ for fixed $(x,y) \in E_A^2,$ then  
		$$ d_\mu = \int_A \frac{d_J}{1+d_J} d\mu(J).$$
	\end{itemize}

\end{Definition}

\begin{Theorem}
	$d^\infty_E$ and $d_\mu$ are pseudo-distances on $E^A$ and on $E_A.$ Moreover $d^\infty_E$ is a distance on 
	$E^A$ and on $E_A.$
\end{Theorem}

\begin{proof}
	We have that $\phi: \R_+ \rightarrow \R_+$ defined by $$ \phi(x) = \frac{x}{1 + x}$$ is concave and therefore, 
	if $d$ is any distance, $\phi \circ d$ is also a distance. This is then an easy exercise to prove that 
	$d^\infty_E$ and $d_\mu$ are pseudo-distances on $E^A$, and hence on $E_A.$  
\end{proof}

{
We notice here that we did not assume that the topology of each $E_J$ is the topology defined by the metric $d_J$. 
Intuitively speaking, this choice seems to be natural. However, {{}it} is not the only one. Let us give an 
example where each $d_J$ does not define the topology of $E_J.$

\begin{example}
    Let $A=\N,$ and $\forall J \in A, E_J=\R$ and let $d_J$ be the discrete distance. Let $\mu$ be the measure on 
    $\N$ such that $\mu(\{n\}) = \frac{1}{(n+1)^2}.$
    Then $d_\mu$ is an ultrametric distance on ${E_A} =\R^\infty$ and $E^A=\R^\N.$ Following \cite{ERMR}, the 
    construction of the diffeology of $E^A$ is possible. In this simple case, this is the product diffeology. 
\end{example}

{Let us construct
$$\mathfrak{B}(Sec(A)) = \left\{ B \subset A \, | \, \exists S \in Sec(A), \, \forall J \in A, (J \in B) 
\Leftrightarrow (\exists K \in S, K < J) \right\}.$$
The set $\mathfrak{B}$ is obviously a base of a filter $\mathfrak{F}$ on $A$ that we call the \textbf{filter of 
sections} of $A$.} 

The filter $\mathfrak{F}$ provides a natural filtration that seems to be related to the ultrametric distance 
$d_\mu$ provided in last example, but the direct relationship with the other distances needs to be analyzed more 
deeply. This precise study is left for future works.}

\section{Diffeologies on profinite families and their limit sets}

Let us fix a projective diffeological family $(E,\pi)$ indexed by a directed set $(A,\leq).$
We consider the product limit set $E^A.$ 
\begin{Definition}
	A diffeology $\p$ on $E^A$ is said to be a profinite diffeology if and only if the canonical injection 
	$$ (E^A,\p) \hookrightarrow \prod_{J \in A} E_J$$ is smooth for the product diffeology of 
	$\prod_{J \in A} E_J.$
\end{Definition}


\begin{Proposition}
	There is a unique profinite diffeology $\p^A$ which is maximal for inclusion. This diffeology coincides with 
	the subset diffeology of $E^A$ in $\prod_{J \in A} E_J.$    
\end{Proposition}
\begin{proof} 
The discrete diffeology {{} ({\em i.e.}, the diffeology for which the only plots are the locally constant maps,} is profinite. Therefore, we can apply Zorn's lemma to the set $\p(E)$ of profinite 
diffeologies {{} equipped with the partial order defined by inclusion}: if we have an increasing chain of profinite diffeologies, then the supremum {{} {i.e.}, the diffeology generated by the union of the diffeologies in the chain,} is also a  profinite 
diffeology, and therefore $\p(E)$ has a maximal element $\p^A.$
{{}Since the subset diffeology $S^A$ is a profinite diffeology, we have $S^A \subset \p^A$. On the other hand, if $p:O_p \rightarrow E^A$ is a plot in $\p^A$, then $\iota \circ p$ is smooth in the product diffeology, and so $p$ is a plot with range $E^A$ that smooth in the product, this is, $p$ belongs to the subset diffeology. Hence $S^A= \P^A$, which ends the proof.}
\end{proof}

Now, since $E_A$ is a subset of $E^A,$ we can equip $E_A$ with the subset diffeology. We get:

\begin{Proposition}
	The profinite diffeology $\p^A$ on $E^A$ defines a diffeology on $E_A$ that we denote by $\p_A.$
\end{Proposition}


\section{Smooth profinite operations and fibrations} \label{s:fibr}

{
\begin{Definition}
    Let $(X,\pi,i)$ and $(X',\pi',i')$ be two profinite diffeological families indexed by $A.$ A smooth map from 
    $X$ to $X'$ is a collection $f = (f_J)_{J \in A}$ of smooth mappings (in the classical sense) 
    $f_J : X_J \rightarrow X'_{J'}$ and an order preserving map $inf(f):A \rightarrow A$ such that {{} we have the following commutative diagram:}

    \begin{center}
    \begin{tikzcd}
X_K \arrow[r, "f_K"] \arrow[d, "\pi_J^K"]
&  X'_{K'}\arrow[d, "{\pi'}_{J'}^{K'}"] \\
X_J \arrow[r, "f_J"]
& X'_{J'}
\end{tikzcd}
\end{center}
    {{} for $J \leq K$ in $A, J' = ind(J)$ and $K'=ind(K).$ }

    We denote by $C^\infty(X,X')$ the set of such families $f = (f_J)_{J \in A}$ of smooth maps.
\end{Definition}

The set
    \[
    S = \left\{C^\infty(X,X') \, | X \hbox{ and } X' \hbox{ are profinite families indexed by } A \right\} 
    \]
    is an associative groupoid for composition, and hence $C^\infty(X,X)$ is an associative monoid for 
    composition, with neutral element. It is then natural to define a diffeology on $C^\infty(X,X')$ by taking 
    plots $p = (p_J)_{j \in A}$, defined on a fixed open subset $O_p$ of an Euclidean space, to be a family 
    $(p_J)_{j \in A}$ such that $$ \forall J \in A, p_J \in C^\infty(O, C^\infty(X_J, X_{ind(p)(J)} )) \hbox{ for 
    the functional diffeology. }$$ 
    Moreover, by virtue of the properties of the functional diffeology, {both the set $S$ and the 
    space $C^\infty(X,X)$}
    carry natural diffeologies compatible with composition. Thus, these two spaces are respectively a 
    diffeological associative groupoid 
    and a diffeological monoid. Finally, we can define the group of diffeomorphisms of the profinite family $X$ 
    as 
$$ 
Diff(X) = \left\{ f \in C^\infty(X,X) \, | \, \exists g \in C^\infty(X,X), g \circ f = f \circ g = Id_X \right\}.
$$

\begin{rem} 
A diffeomorphism of the profinite family $X$ does not define index-wise a diffeomorphism $f_J \in Diff(X_J),$ 
$\forall J \in A .$ As a counterexample, let us consider the profinite family composed of four manifolds indexed 
by $A=\{I,J,K,L\}:$
\begin{itemize}
    \item $X_I = \{(0,0)\}$
    \item $X_J = \{(x,0) \, | \, x \in \R \}$
    \item $X_K = \{(0,y) \, | \, y \in \R \}$
    \item $X_L = \{(x,y) \, | \, x \in \R \hbox{ and } y \in \R\},$
\end{itemize}
where the projections and injections are the canonical ones, and the order on $A$ is the order induced by 
inclusion of subsets. Then the orthogonal symmetry on $X_L$ with respect to $\{x=y\}$ defines a profinite 
(involutive) diffeomorphism $f$ such that $f_J : X_J \rightarrow X_K$ and $f_K : X_K \rightarrow X_J.$
\end{rem}

\begin{Lemma}
    Let $(X_\alpha,\pi_\alpha,i_\alpha)_{\alpha \in \Lambda}$ be a set, indexed by the finite set $\Lambda,$ of 
    profinite diffeological families indexed by $A.$ Then the product projections 
    $\pi_{J}^K = \prod_{\alpha \in \Lambda} (\pi_{\alpha})_J^K $ and the product injections 
    $i_{K}^J = \prod_{\alpha \in \Lambda} (i_{\alpha})_K^J, $ for $J \leq K$ in $A,$ define a profinite 
    diffeological product family $$\left(\prod_{\alpha \in \Lambda}X_\alpha,\pi,i\right)$$ and 
    \begin{itemize}
    \item $\left(\prod_{\alpha \in \Lambda} X_\alpha \right)^A = 
    \prod_{\alpha \in \Lambda} \left(\left( X_\alpha\right)^A\right) ,$
    \item $\left(\prod_{\alpha \in \Lambda} X_\alpha \right)_A = 
    \prod_{\alpha \in \Lambda} \left(\left( X_\alpha\right)_A \right),$ \end{itemize}
\end{Lemma}
\begin{proof}
    Direct.
\end{proof}

This result means that profinite diffeological monoids and groupoids can be defined in general, and therefore we 
can consider families of groups, modules, fields, rings, algebras and other straightforward specifications of 
algebraic structures. However we propose here a restricted class of profinite objects, that we call \textbf{tame} 
in case of necessity to distinguish them from more complex structures. 


\begin{Lemma} \label{lemma-monoid}
 Let $(M, \pi, i)$, $M= \{ M_J \}_{J \in A}$, be a profinite diffeological family indexed by $A,$ such that 
 $\forall J \in A,$ $M_J$ is a 
 diffeological monoid (resp. associative monoid). If $\pi$ is a family of morphisms of diffeological monoids, then 
 $M^A$ and $M_A$ are diffeological monoids (resp. associative monoids).    
\end{Lemma}

\begin{proof}
    Let us start with $(x,y) \in M^A \times M^A$ 
    and let $\vdash$ 
    denote the family of monoid laws, i.e. $\vdash\; = (\vdash_J)_{J \in A}$ and  $\forall J \in A,$ the monoid 
    law of $M_J$ is $\vdash_J.$
    Let us consider $x \vdash y,$ that is, 
    let us consider the family of operations $\{ x_J \vdash_J y_J \, | \, J \in A \}.$ 
    Since $\forall (J,K) \in A^2$, $J \leq K$, the map $\pi_J^K$ is a morphism, we have that 
    $\pi_J^K(x_K \vdash_K y_K) = x_J \vdash_J y_J.$ 
    Therefore, $\vdash$ is a well-defined operation $M^A \times M^A \rightarrow M^A.$
    This shows the structure of monoid of $M^A,$ and the same arguments hold for the structure of $M_A.$
\end{proof}
}
\begin{Theorem} \label{thm65}
	Assume that 
	\begin{enumerate}
		\item There exists a profinite family of rings $(\mathbf{R}, \pi, i)$ indexed by {{} a directed set} $A$ such that the 
		projection maps $\pi$ and the injection maps $i$ are morphisms of rings
		\item The family $(E,\pi)$ is a projective family of modules over $\mathbf{R},$ that is, 
		$\forall J \in A,$ $E_J$ is a $\mathbf{R}_J-$module. 
	\end{enumerate}
Then 
\begin{enumerate}
    \item 
    $\mathbf{R}^A$ is a ring, and $E^A$ is a { $\mathbf{R}^A$-module}. If, moreover, the algebraic 
    structures of each $J \in A$ are diffeological, then $E^A$ is a diffeological 
    { $\mathbf{R}^A$-module.}
    \item The same hold for $\mathbf{R}_A$, and $E_A$ is a $\mathbf{R}_A$-module. If moreover the algebraic 
    structures of each $J \in A$ are diffeological, then $E_A$ is a diffeological 
    {$\mathbf{R}_A$-module.} 
\end{enumerate}
\end{Theorem}

\begin{proof}
	{The sets $E^A$ and $\mathbf{R}^A$ are already well-defined as projective limits. Only the 
	{limits} of the algebraic structures have to be checked to be well-defined. 
    
    For this, Let us start with $(x,y) \in X \times Y$ with 
    $$(X,Y) \in \{(\mathbf{R}_A, \mathbf{R}_A),(\mathbf{R}^A, \mathbf{R}^A), (\mathbf{R}_A, E_A),
    (\mathbf{R}^A, E^A),(E_A, E_A),(E^A, E^A)  \}.$$ 
    If $X=Y,$ then Lemma \ref{lemma-monoid} applies to all the monoidal structures under consideration. If 
    $(X,Y) \in \{(\mathbf{R}_A, E_A),(\mathbf{R}^A, E^A)\}$ we have to check the compatibility under the profinite 
    sequence of projections under the same procedure as before. Let $\vdash$ 
    denote the family of monoid laws, i.e. $\vdash = (\vdash_J)_{J \in A}$ such that  $\forall J \in A,$ the law 
    of $M_J$ is $\vdash_J.$
    Let us consider $x \vdash y,$ that is, 
    let us consider the family of operations $\{ x_J \vdash_J y_J \, | \, J \in A \}.$ 
    Since $\forall (J,K) \in A^2, J \leq K \Rightarrow \pi_J^K$ is a morphism, then 
    $\pi_J^K(x_K \vdash_K y_K) = x_J \vdash_J y_J.$ 
    This shows the structure of $\mathbf{R}^A-$ module of $E^A,$ and  the structure of $\mathbf{R}_A-$ module of 
    $E_A.$}
\end{proof}
{
\begin{Lemma}
 Let $(G, \pi, i)$ be a profinite diffeological family indexed by $A,$ such that $\forall J \in A,$ $G_J$ is a diffeological group. If $\pi$ is a family of morphisms of diffeological groups, then $G^A$ and $G_A$ are diffeological groups.    
\end{Lemma}
\begin{proof}
We have to check the existence of the group law and of the inversion mapping on $G^A$ and on $G_A.$ The group law 
is a consequence of Lemma \ref{lemma-monoid}. For inversion, let us consider $x = (x_J)_{J \in A} \in G^A$ (resp. 
$G_A$) and let us consider $y = (x_J^{-1})_{J \in A}.$ Obviously, for the product monoid law, $yx = xy = 1$ and 
the map $x \in G^A \, {{}(\hbox{resp. }G_A)} \mapsto y \in \prod_{J \in A} G_J$ is smooth. Therefore, we only have to prove that 
$y \in G^A,$ which is direct from the statement that the family $\pi$ is a family of morphisms of diffeological 
groups.
\end{proof}
\begin{Cor}
	The same conclusion of Theorem $\ref{thm65}$ holds for diffeological algebras, diffeological 
	(non-commutative) fields, diffeological groups, diffeological vector spaces.
    \end{Cor}
\begin{proof}
    {All these settings are specifications of the notion of module over a ring and of the notion of group.} 
\end{proof}

\begin{Definition}
Let $G$ and $X$ be two profinite families indexed by $A,$ such that $G$ is a profinite family of diffeological Lie 
groups. An action of $G$ on $X$ is a morphism of diffeological monoids from $G$ to $C^\infty(X,X).$
\end{Definition}

Let us now assume that the spaces $E_J$ {of} the profinite family $(E,\pi)$ are all total spaces of 
$F_J-$fiber bundles $p_J: E_J \rightarrow M_J$ where $(F,\pi)$ and $(M,\pi)$ are also profinite families, with 
smooth compatibility of the fiber structures with the projective structure, {in other words, such that the family 
of projections $\pi$ and the family of injections $i$ are families of morphisms of diffeological fiber bundles.}

\begin{Theorem}
	Under these conditions, $E^A$ (resp. $E_A$) is a smooth fiber bundle over $M^A$ (resp. $M_A$) with typical 
	fiber $F^A$ (resp. $F_A$).
\end{Theorem}

\begin{proof}
	{Direct.}
\end{proof}

\section{Smooth cylindrical functions and regular cotangent space}

{{} The notion of cyindrical map is central in \cite{AM99}. Indeed, in our vocabulary, cylindrical functions provide a natural generating family for the Fr\"olicher structure over the inductive limit or the projective limit of a profinite family. we settle the existing definitions of \cite{AM99} in our diffeological vocabulary, which precises the conceptual meaning of the constructions, and complete these notions for our needs and withing our actualized state of knowledge.}

\begin{Definition}
Let us assume that $(M,\pi,i)$ is a diffeological profinite family. A map $f \in C^\infty(E^A,\R)$ (resp. 
$f \in C^\infty(E_A,\R)$) is called \textbf{cylindrical} if there exists $\mathcal{S}_f \in Sec(A)$ such that
	$$\forall (x_I)_{I \in A}, \quad f((x_I)_{I \in A} ) = f((y_I)_{I\in A}) $$ where $(y_I)_{I\in A}$ is an 
	element of $E^A$ (resp. $E_A$) such that 
	$$\forall I \in A, y_I= \left\{ \begin{array}{lcl}
	x_I & \hbox{ if } & \exists S \in \mathcal{S}_f, I \leq S \\
	i_I^S \circ \pi_S^I (x_I) & \hbox{ for } & S \in \mathcal{S}_f \hbox{ such that } S \leq I \end{array} 
	\right.$$
We note by $Cyl(E^A)$ (resp. $Cyl(E_A)$) the set of cylindrical functions on $E^A$ (resp. on $E_A$)
\end{Definition}

\begin{rem}
	It is clear that the choice of $(y_I)_{I\in A}$ in $E_A$ is not unique.
	Indeed, the definition of the map $$ p_{xy}:(x_I)_{I \in A} \mapsto (y_I)_{I\in A}$$ depends on the chosen 
	section $S_f.$ For example, consider $E = (\R^n)_{n \in \N^*}$ with coordinate-wise projections and a 
	cylindrical function $f$ which extends $f_1 \in C^\infty(\R,\R).$ Then, even in this elementary case, one can 
	define $\mathcal{S}_f $ as either $\R$ or $\R^2...$ or $\R^k$ for any $k \in \N^*$ and for all these sections, 
	\begin{itemize}
		\item $f$ will be cylindrical
		\item the map $p_{xy}$ will depend on the chosen section $\mathcal{S}_f.$
	\end{itemize}
\end{rem}

{\begin{rem}
    We also remark that the mappings $p_{xy},$ that depend on a section {{} $\mathcal{S}_f,$} are 
    defined by the family $(x_J)_{J \in \overline{B}}$ where 
    $${{} 
    \overline{B} = \{ J \in A \, | \, \exists K \in \mathcal{S}_f, K < J \} 
    }$$ 
    is such that $\overline{B} \in \mathfrak{F}$, {the filter of sections of $A$}. 
\end{rem}}

These comments mean that we have the following, {{}which has to be compared with the essence of \cite{AM99}: }
\begin{Theorem}
	\begin{enumerate}
		\item $Cyl(E^A) = Cyl(E_A)$
		\item 
        $Cyl(E^A)$ separate $E^A.$
        \item {{}$\p_A$, resp. $\p^A$ is a reflexive diffeology. The corresponding Fr\"olicher structure is generated by $Cyl(E_A).$ }
	\end{enumerate}
\end{Theorem}

\begin{proof}
	\begin{enumerate}
	    \item We have $E_A \subset E^A$ therefore $Cyl(E^A) \subset Cyl(E_A).$ Let $f \in Cyl(E^A)$ with section 
	    $\mathcal{S}_f \in Sec(A)$ and let $(x_J)_{J \in \overline{B}}.$ Then 
	    $p_{xy}((x_J)_{J \in \overline{B}}) \in E^A.$ Therefore, $f$ defines a map on 
	    $E_{\overline{B}} \subset E_A$ by $$f((x_J)_{J \in \overline{B}}) = 
	    f \circ p_{xy}((x_J)_{J \in \overline{B}}).$$ This correspondence is obviously univoque.
        \item Projection maps $E^A \rightarrow E_J$ separate $E^A.$
        \item {{}Let $I \in A.$ Then $C^\infty(E_I,\R)$ is a set of generating functions for the nebulae diffeology of $E_I$, and it is also a subset of $Cyl(E_A).$ Therefore, $$p \in \p(E_I) \Leftrightarrow \forall f \in C^\infty(E_I,,\R), f \circ p \in C^\infty(\R,\R)  \Leftrightarrow  \forall f \in Cyl(E_A), f \circ p \in C^\infty(\R,\R).$$}
	\end{enumerate}
	\end{proof}

{{} Let us now turn to new results.}
\begin{Theorem}
	
	The set of polynomials of cylindrical maps on $E_A,$ {{} definied by 
	$$\bigcup_{n \in \N^*}\left\{ P \circ (f_1,...,f_n) \, | \, P \in \R[X_1,...X_n] \hbox{ and } (f_1,...f_n) \in Cyl(E_A)^n\right\},$$}
	 is dense in $C^0(E^A,\R)$ if $E^A$ is compact and it is  
	equipped with the distance $d^\infty_E.$ 
\end{Theorem}

\begin{proof}
	This is a consequence of \cite[X, Proposition 5]{Bou-Top}.
\end{proof}

Now we can turn to cotangent {{} fiber bundles}. We remark that, for fixed index 
$J \in A,$ $T^*E_J$ is the (classical) contangent space of the manifold $E_J.$ Moreover, if $J \leq K,
$ the projection map $\pi_J^K$ defines a smooth map from  $T^*E_J$ to $T^*E_K$ by 
$$ \alpha \in T^*E_J \mapsto \alpha \circ d \pi_J^K$$ that we also note by $\pi_J^K,$ and the injection map 
$i_K^J$ defines a smooth map from  $T^*E_K$ to $T^*E_J$ by 
$$ \alpha \in T^*E_K \mapsto \alpha \circ d i_K^J$$ that we also note by $i_K^J.$   

\begin{Proposition}
	$(T^*E,i)$ is a profinite family with injection maps $\pi.$
	\end{Proposition}

\section{Tangent space}
We stated in Section 1 that there is more than one way to define the tangent space of a diffeological space. {{} in this section we show that in the profinite case these previous definitions coincide if the index set $A$ satisfies a technical condition.} In the
profinite case we have:
 
\begin{Proposition}
	If $(E,\pi)$ is a profinite family, then the family of tangent spaces of finite dimensional manifolds 
	$(TE,T\pi)$ is a diffeological projective family. 
\end{Proposition}
\begin{proof}
	Let $K \in A.$ Then $TE_K$ is a finite-dimensional manifold and $\forall (J,K) \in A^2,$ if $J \leq K, $
	$$ d \pi^K_J \circ di^J_K = Id_{TE_K}.$$
	All maps $d \pi^K_J$ and $di^J_K$  are smooth for the underlying nebulae diffeologies in the tangent spaces 
	$TE_J$ and $TE_K.$  
\end{proof}

\begin{Definition}
    We say that $A$ is \textbf{finitely cylindrical} if the filter {
    	 of sections} $\mathfrak{F}$ is 
    generated by the base $$\mathfrak{B}_{finite} = 
    \{ \overline{B} \in \mathfrak{B} \, | \, \overline{B} \hbox{ is a finite subset of } A \},$$
    {in which $\mathfrak{B}$ is the basis of $\mathfrak{F}$ defined in Section 3.}
\end{Definition}

\begin{Theorem} \label{th:tangent}
Let us now assume that $A$ is finitely cylindrical.  Then, the  internal tangent bundle and the external tangent 
space coincide with the limit set and the product set of the family $(TE,D\pi),$ that we therefore note $TE_A$ and 
$TE^A$ without any ambiguity.
\end{Theorem}

\begin{proof}
	Let $f \in C^\infty(E_A,\R)$ and let $d$ be a derivation in $C^\infty(E^A,\R).$ Let us consider a cylindrical 
	map $f.$ Then $f$ is identified with a smooth map on the smooth manifold 
	$\prod_{J \leq S, S \in \mathcal{S}_f} E_J.$
	Therefore, there exists a unique vector field $X_{d,J}$ on  $\prod_{J \leq S, S \in \mathcal{S}_f} E_J$ such 
	that $$ d f =  X_{d,J}(f).$$
	Therefore, we have an identification 
between the internal tangent cone and the external tangent space of $\prod_{J \leq S, S \in \mathcal{S}_f} E_J$ 
which extends to $E_A$ and to $E^A.$ 	 
\end{proof}

\section{cylindrical Riemannian metrics, differential forms, presymplectic and symplectic structures}
\subsection{Riemannian and Hermitian metrics}

We define Riemannian metrics by means of the following proposition.

\begin{Proposition}
   A \textbf{Riemannian metric} on a profinite vector bundle {{} in the sense of section \ref{s:fibr}} is defined equivalently by 
   \begin{itemize}
       \item a smooth map $$g_A: {(E_A)^{(2)}}  \rightarrow \R$$ which is fiberwise a symmetric positive definite  
       bilinear map;
       \item a family of smooth maps $$g_I: {E_I^{(2)}}  \rightarrow \R\; , \quad I \in A$$ which are (fiberwise) 
       symmetric positive definite  bilinear maps, and that satisfy the compatibility condition by pull-back: 
       $\forall (J,K) \in A^2,$ if $J < K,$ $g_I = (i^I_K)^* (g_K)\; ;$
       \item a smooth map $$q_A: E_A \rightarrow \R$$ which is fiberwise a quadratic positive definite  map; 
       \item a family of smooth maps $$q_I: E_I  \rightarrow \R\; , \quad I \in A$$ which are fiberwise quadratic 
       positive definite, and that satisfy the compatibility condition by pull-back: $\forall (J,K) \in A^2,$ 
       if $J < K,$ $q_I = (i^I_K)^* (q_K)$.
   \end{itemize}
\end{Proposition}
\begin{proof}
The straightforward correspondence between the four items arises from the fact that they all generate a family of 
smooth maps $$q_I: E_I \times E_I \rightarrow \R$$ which are (fiberwise) quadratic positive definite maps, and 
hat 
satisfy the compatibility condition by pull-back: $\forall (J,K) \in A^2,$ if $J < K,$ $q_I = (i^I_K)^* (q_K).$ 
Reciprocally, within this family, one can build directly the full picture by classical arguments. 
\end{proof}

\begin{Proposition}
   A \textbf{Hermitian metric} on a profinite vector bundle is the same as before but for Hermitian metrics.
   \end{Proposition}

\begin{Proposition}
   A \textbf{pseudo-Riemannian metric} on a profinite vector bundle is the same as before but for 
   pseudo-Riemannian metrics. 
   \end{Proposition}

   \begin{Proposition}
   A \textbf{pseudo-Hermitian metric} on a profinite vector bundle is the same as before but for pseudo-Hermitian 
   metrics.
   \end{Proposition}

   \begin{Definition}
       A pseudo-Riemannian (resp. a pseudo-Hermitian) metric $$(g_I)_{I \in A}$$ on a profinite vector bundle is 
       positive definite if and only if $\forall I \in A, g_I$ is positive definite on $E_I.$
   \end{Definition}

   \begin{rem}
       In this subsection we did not deal with \emph{non-degenerate} pseudo-Riemannian or pseudo-Hermitian 
       metrics. This point will be discussed in a specific section.
   \end{rem}
   
\subsection{Tame differential forms}
The space of differential forms $\Omega^*(E_I)$ is the (classical) de Rham space of the (classical) manifold 
$E_I, $ therefore, by pull-back mappings, we have natural embeddings
$$ (\pi_I^K)^* : \alpha \in \Omega^*(E_I) \mapsto (\pi_I^K)^* \alpha \in \Omega^*(E_K)$$
and a projection
$$ (i_K^I)^* : \alpha \in \Omega^*(E_K) \mapsto (i_K^I)^* \alpha \in \Omega^*(E_I)$$

\begin{Theorem}
	The family $(\Omega^*(E_I),(i_K^I)^*,(\pi_I^k))$ is a profinite family.
\end{Theorem}
\begin{proof}
For fixed indexes $(I,K) \in A^2,$ if $I \leq K;$ the following diagram is commutative.  
 \begin{center}
  \begin{tikzcd}
 \Omega^*(E_I) \ar{rr}{(\pi_I^K)^*}  & & \Omega^*(E_K) \\
   & \Omega^*(E_K) \ar{ul} {(i_{K}^{I})^*} \ar[swap]{ur}{id} &
  \end{tikzcd}
  \end{center}
Then, we remark that $\forall (I,J,K) \in A^3,$ if $I \leq J \leq K, $ 
$(i_{K}^{J} i_J^I)^* =  (i_J^I)^* (i_{K}^{J})^*$ and $(\pi_{I}^{J} \pi_J^K)^* =  (\pi_J^I)^* (\pi_{K}^{J})^*.$
This ends the proof.
\end{proof}

It is a classical fact that the notion of closed or exact form is not invariant under projections. Therefore, we 
are led to define the following objects:

\begin{Definition} \label{d:tamediff}
Let $n \in \N.$ We define the space of tame $n-$forms $\Omega_t^n(E_A)$ as the product limit space of the family 
$(\Omega^n(E_I),(i_K^I)^*,{{}(\pi_I^K)^*}).$
\end{Definition}
By standard properties of the pull-back,
    the de Rham differential on each algebra $\Omega^*(E_I)$ extends to an exterior differential on 
    $\Omega_t^n(E_A).$

\section{Non-degeneracy {{} on profinite spaces}}

\begin{Definition}
A pseudo-Riemannian metric $g=(g_I)_{I \in A}$ (resp. a tame 2-form $(\omega_I)_{I \in A}$) is \emph{{{}projectively} 
non-degenerate}
if $\forall I \in A,$ $g_I$ (resp. $\omega_I$) is non-degenerate (in the classical sense).
\end{Definition}

\begin{Definition}
A pseudo-Riemannian metric $g$ on $E_A$ (resp. a tame 2-form $\omega$ on $E_A$ is 
\emph{weakly non-degenerate} if {{} 
$\forall u \in TE_I \mbox{ such that } u \neq 0, \exists J \geq I, \exists v \in TE_J, g_J(D\iota^I_J(u),v) \neq 0$ (respectively, $
\omega_J(D\iota^I_J(u),v) \neq 0$) .}
\end{Definition}
\begin{Proposition}
    ``projectively non-degenerate'' implies ``weakly non-degenerate''.
\end{Proposition}

\begin{example}
The reciprocal statement is false.
    Let $A=\{1,2\}$and let $E_I = \R^I$ for $I \in A.$ Then the (constant) pseudo-Riemannian metric $$\left(\begin{array}{cc} 0 & 1 \\1 & 0 \end{array}  \right)$$
    is weakly non-degenerate for $TE_A = TE_2$ but not projectively non-degenerate for the profinite family $(E_1,E_2).$
\end{example}

\section{Symplectic geometry}

{
\begin{Definition}
A tame presymplectic form on the profinite family $(E, \pi,i)$ is a tame closed 2-form 
$\omega =(\omega_I)_{I \in A} \in \Omega^2_t(E_A)$ such that $\omega_I$, $I \in A$, is of constant
rank. If moreover $\omega$ is projectively non-degenerate, we call it tame symplectic form.  
\end{Definition}
{}
\begin{rem} The condition projectively non-degenerate is more restrictive than weakly non degenerate, even for symplectic forms. For example, let us define $\omega = \sum_{i \in \N} dx_{2i} \wedge dx_{2i+1}$ on $\R^\infty $ which is the inductive limit of $\{\R^n, n \in \N\}.$ Then $\omega$ is degenerate on each $\R^{2i+1},$ while this pathology does not appear with the family $\{\R^{2n}, n \in \N\}$ which also defines $\R^\infty$ as an inductive limit.  
\end{rem}}

If $E_A$ is equipped with a {{} non-degenerate} Riemannian metric $(g_I)_{I \in A},$ then for each 
$J \in A$ we can build $\mathfrak{I}_J \in End(T E_J)$ such that 
$$ \omega_J(\cdot\, , \cdot) = g_I(\cdot\, , \mathfrak{I}_J(\cdot))\; .$$

\begin{Definition}
{{} Let $\omega$ be a tame symplectic form.}
The Hamiltonian vector field corresponding to a function $H \in \Omega^0_t(E_A)$ is a collection of smooth vector 
fields 
$$(X_{H,J})_{J \in A} \in \prod_{J \in A} \mathcal{X}(TE_J)$$ 
such that $$\omega_J(X_{H,J}\,, \cdot) = dH_J $$ 
for all $J \in A .$
\end{Definition}

\begin{Theorem}
    If $f \in \Omega^0_t(E_A),$ then the Hamiltonian vector field defines a smooth section of the tangent vector space $T E^A.$
\end{Theorem}

\begin{proof}
By compatibility conditions.
\end{proof}}

{ Earlier in section \ref{section:vpb} we considered the notion of a diffeological action of a group $G$ on a 
diffeological space $E.$ Let $\omega$ be a symplectic form on $E.$ {{} 
	A profinite Lie group 
	$G = (G_I)$ has a profinite symplectic action on a profinite symplectic space $E$ if each $G_I$act on each $E_I$ "symplectically", that is, 
	$\phi^*_{I,g}\omega_I = \omega_I$ for all $I \in A$ and all $g \in G_I$. Then the} action of a Lie group $G$ on $E$ is called 
Hamiltonian action with respect to $\omega,$ if for every $\xi$ in the Lie algebra {{}$\mathfrak{G}$} there exists a smooth function 
$$ \mu(\xi):E\longrightarrow\mathbb{R}$$
such that the infinitesimal action of $\xi$ on $E$ (i.e. $X_\xi=\frac{d}{dt}(exp(t \xi))_E\mid_{t=0}$) is Hamiltonian, in other words
$$X_\xi=X_{\mu(\xi)},$$ where $X_{\mu(\xi)}$ is the Hamiltonian vector field determined by 
$\omega(X_{\mu(\xi)},-)=d\,\mu(\xi).$ 

\begin{Definition}
For every $x\in E,$ the mapping $\mu(x):{{}\mathfrak{g}}\rightarrow\mathbb{R}; \ \ (\mu(x))(X)=(\mu(X))(x)$ is a linear form on {{}$\mathfrak{g}$} and defines a smooth mapping
\begin{eqnarray*}
\mu:E\longrightarrow{{}\mathfrak{g}^*} \\
x\longmapsto \mu(x)
\end{eqnarray*}
which is called the \textbf{momentum mapping} of the Hamiltonian action $G$ on $E.$
\end{Definition}
}

	\section{Elementary examples}
	\subsection{Geometry on matrices}
	Through the canonical inclusions 
	$$\R \subset \R^2 \subset \cdots \subset \R^n \subset \cdots$$
	let us consider the corresponding inclusion sequence of square matrices
	$$ \R \subset M_2(\R) \subset \cdots \subset M_n(\R) \subset \cdots$$
	define a profinite family of spaces 
	$$ E_n = M_n(\R)$$ indexed on $a = \N^*.$
	\begin{Theorem}
		Let us consider $(e_n)_{n \in \N^*}$ an orthonormal base of a Hilbert space $H.$ Then $E_A$ can be identified with an algebra of finite rank operators on $H,$ different from the algebra of finite rank operators on $H.$
	\end{Theorem}  

\begin{proof}
	{$E_A$ is obviously an algebra of finite rank operators, but given $x \in H - \R^\infty,$ the orthogonal projection on $\R x$ is a rank 1 operator which is not to $E_A$ .}
\end{proof}


The same way, since $A$ is totally ordered, $E^A$ is an algebra. {We can also make the same constructions by replacing $\R$ by $\mathbb{C}.$} Let us show that $E^A$ is far from being identifiable with some algebra of bounded operators.
\begin{Theorem}
	Let $H = L^*(S^1,\mathbb{C}),$ equipped with the Fourier base indexed by $\N^*.$ If $\Delta$ is the positive Laplacian, then $e^\Delta \in E^A.$ 
\end{Theorem}

\begin{proof}
	{The operator $\Delta$ is diagonal in the Fourier base and hence each space $\R^n$ is stable under $\Delta$. Therefore, $\Delta$ and $e^\Delta$ are diagonal, bounded and well-defined on each $\R^n,$ therefore $e^\Delta \in E^A.$ }
\end{proof}
	\subsection{Geometry on Jet spaces}
    {Reformulating {{}\cite{GP2017} together with \cite{MR2024-2}}, the jet space $J^\infty(M,E)$ is the product set of the family $X=\left\{J^k(M,E) \, | \, k \in \N^* \right\},$ where the projections $\pi_k^{k+p}:J^{k+p}(M,E) \rightarrow J^k(M,E), \forall (k,p) \in (\N^*)^2,$ are forgetful mappings and the injections $i^k_{k+p}:J^{k}(M,E) \rightarrow J^{k+p}(M,E), \forall (k,p) \in (\N^*)^2,$ are defined through the 0-sections of the vector bundle $J^{k+p}(M,E)$ over its base $J^{k}(M,E).$ Since this profinite family is indexed by $A=\N,$ one can see that the statements in \cite{MR2024-2} about $TJ^\infty(M,E) = TX^A$ in our notations is a specification of Theorem \ref{th:tangent}.}
	
 \subsection{The Wiener space as a diffeological space}

Let \( W = C_0([0,1], \mathbb{R}^n) \) denote the classical Wiener space, i.e., the Banach space of continuous functions \( \gamma : [0,1] \to \mathbb{R}^n \) with \( \gamma(0) = 0 \), endowed with the topology of uniform convergence.

We endow \( W \) with the \emph{cylindrical diffeology}, generated by the collection of evaluation maps:
\[
\pi_{t_1, \dots, t_k} : W \to \mathbb{R}^{nk}, \quad \gamma \mapsto (\gamma(t_1), \dots, \gamma(t_k)),
\]
for arbitrary finite sequences \( 0 < t_1 < \cdots < t_k \leq 1 \). 
Therefore, within our terminology, the index set is given by $A = \mathfrak{P}_f([0,1]).$
Given $K \in A,$ we have $W_K = \R^{n|K|}$ and there are canonical projections $W \rightarrow W_K$ given by the canonical evaluation maps. This family $(W_K)$ form a profinite diffeological space

\subsubsection*{Cotangent structure and differential forms}

The cotangent space \( T^*_\gamma W \), at a path \( \gamma \in W \), is defined as the colimit
\[
T^*_\gamma W = \varinjlim_{(t_1, \dots, t_k)} T^*_{\gamma(t_1), \dots, \gamma(t_k)} \mathbb{R}^{nk},
\]
over the system of evaluation maps, in other words, $T^*W$ is the regular cotangent space of the family $(W_K).$ A cotangent vector at \( \gamma \) is thus represented by a compatible family of finite-dimensional covectors \( \eta_{t_1, \dots, t_k} \), defined modulo pullbacks by refinement of time partitions.

A smooth differential 1-form on \( W \) is a compatible collection \( \{ \alpha_p \in \Omega^1(U) \} \), one for each plot \( p : U \to W \), satisfying the usual pullback condition under smooth reparametrizations. Such a form can be written locally as the pullback:
\[
\alpha = \pi_{t_1, \dots, t_k}^* \omega,
\]
for some 1-form \( \omega \in \Omega^1(\mathbb{R}^{nk}) \), and is called a \emph{cylindrical 1-form}. In explicit terms, a 1-form \( \alpha \) evaluated on a tangent vector \( h \in W \) at a path \( \gamma \in W \) takes the form:
\[
\langle \alpha, h \rangle = \sum_{i=1}^k \xi_i \cdot h(t_i),
\]
for some finite family of times \( t_i \in [0,1] \) and covectors \( \xi_i \in (\mathbb{R}^n)^* \).

Higher-degree differential forms on \( W \) are constructed analogously: a smooth \( r \)-form is locally a pullback
\[
\pi_{t_1, \dots, t_k}^* \omega^{(r)}, \quad \omega^{(r)} \in \Omega^r(\mathbb{R}^{nk}).
\]
The collection \( \Omega^\bullet(W) \) of all such forms defines a complex with differential \( \mathrm{d} \) induced from the standard de Rham complex on \( \mathbb{R}^{nk} \). Within all these definitions, the space of so-called cylindrical differential forms (in the usual sense) is exactly the space of tame differential forms $\Omega^*(W)$ in the sense of Definition \ref{d:tamediff}. 

\medskip

\paragraph{\bf Acknowledgements:}
J-P.M. acknowledges the France 2030 framework programme Centre Henri Lebesgue ANR-11-LABX-0020-01 for 
creating an attractive mathematical environment. 
E.G.R.'s research was partially supported by the FONDECYT grant \#1241719.

\end{document}